\documentclass[11pt]{amsart}

\usepackage[colorlinks,linkcolor=blue,anchorcolor=magenta,citecolor=red,CJKbookmarks=True]{hyperref}
\usepackage{lineno,hyperref}
\modulolinenumbers[5]
\usepackage{amsmath}
\usepackage{amssymb}
\usepackage{mathrsfs}
\usepackage{amsthm}
\usepackage{bm}
\usepackage{graphicx}
\usepackage{booktabs}
\usepackage{float}
\usepackage{subfig}
\usepackage{geometry}
\usepackage{bbm}
\usepackage{color}
\usepackage{upgreek}

\newtheorem{Def}{Definition}[section]
\newtheorem{lem}[Def]{Lemma}
\newtheorem{tho}[Def]{Theorem}
\newtheorem{pro}[Def]{Proposition}
\newtheorem{rem}[Def]{Remark}

\newtheorem{cor}[Def]{Corollary}
\newtheorem{hyp}[Def]{Assumption}

\newcommand{\RNum}[1]{\uppercase\expandafter{\romannumeral #1\relax}}
\newcommand{\ud}{\mathrm d}
\allowdisplaybreaks[4]

\begin{document}

\title[]
{A splitting semi-implicit method for stochastic incompressible Euler equations on $\mathbb T^2$}

\author{Jialin Hong}
\address{Academy of Mathematics and Systems Science, Chinese Academy of Sciences, Beijing
 100190, China; School of Mathematical Sciences, University of Chinese Academy of
 Sciences, Beijing 100049, China}
\email{hjl@lsec.cc.ac.cn}

\author{Derui Sheng}
\address{Academy of Mathematics and Systems Science, Chinese Academy of Sciences, Beijing
 100190, China; School of Mathematical Sciences, University of Chinese Academy of
 Sciences, Beijing 100049, China}
\email{sdr@lsec.cc.ac.cn}
\author{Tau Zhou}
\address{Academy of Mathematics and Systems Science, Chinese Academy of Sciences, Beijing
 100190, China; School of Mathematical Sciences, University of Chinese Academy of
 Sciences, Beijing 100049, China}
\email{zt@lsec.cc.ac.cn (Corresponding author)}

\thanks{This work is supported by National Natural Science Foundation of China (No. 11971470, No. 11871068, No. 12031020, No. 12022118).}

\date{\today}

\dedicatory{}

\begin{abstract}

The main difficulty in studying numerical method for stochastic evolution equations (SEEs) lies in the treatment of the time discretization (\cite{MR1873517}). Although fruitful results  on numerical approximations for SEEs have been developed, as far as we know, none of them include that of stochastic incompressible Euler equations. To bridge this gap, this paper proposes and analyses a splitting semi-implicit method in temporal direction for stochastic incompressible Euler equations on torus $\mathbb{T}^2$ driven by an additive noise. By a Galerkin approximation  and the fixed point technique, we establish  the unique  solvability of the proposed method. Based on the regularity estimates of both exact and numerical solutions, we measure the error in $L^2(\mathbb{T}^2)$ and show that the pathwise convergence order is nearly $\frac{1}{2}$ and the  convergence order in probability is almost $1$.
\end{abstract}

\maketitle

\section{Introduction} 
In this paper, we consider the stochastic incompressible Euler equations on torus $\mathbb{T}^2:=\mathbb{R}^{2}/\mathbb{Z}^{2}$ 
\begin{equation}
\left\{\begin{array}{l}
\mathrm{d} u+(u \cdot \nabla) u \mathrm{d} t+\nabla \bm{\pi} \mathrm{d} t=\mathrm{d} W, \ t\in (0,T],\\
\nabla \cdot u=0
\end{array}\right.\label{Euler}
\end{equation}
with $T>0$ and an initial value $u(0)$ in appropriate function spaces. Here, the velocity field $u$ and the pressure scalar field $\bm{\pi}$ describe the motion of an incompressible, inviscid, homogeneous fluid, and $W$ is a Hilbert space valued Wiener process on a complete filtered probability space $(\Omega, \mathcal{F},\{\mathcal{F}_{t}\}_{t\geq 0}, \mathbb{P})$. As a classical model in fluid mechanics theory, Eq. \eqref{Euler} is closely related to turbulence theory (see e.g. \cite{MR1245492}), where the random external force is employed to simulate uncertainties from environments. There are a plenty of works on the mathematical theory of two-dimensional stochastic Euler equations. For instance, we refer to \cite{MR1674779,MR3483892,MR1880243,MR1722278,MR3161482} for the results on the well-posedness for the multiplicative noise case, to \cite{MR4107937,MR2213944} for the existence of invariant measures for the additive noise case, and to \cite{MR3947288} for the study of the associated Kolmogorov equation for the transport noise case. However, as far as we know, there is few results about the numerical approximations of stochastic Euler equations. As stated in \cite{MR1873517}, the main difficulty in studying numerical method for stochastic evolution equations (SEEs) lies in the treatment of the time discretization. This motivates us to construct a temporal semi-discretization approximation for Eq. \eqref{Euler} and to explore its convergence rate.

Owing to the absence of explicit expressions of true solutions, extensive literature on numerical approximations for SEEs has emerged in recent years. Compared with the case of global Lipschitz nonlinearity, the convergence analysis of numerical methods for SEEs with non-global Lipschitz coefficients, such as stochastic Euler equations and stochastic Navier--Stokes equations, is more challenging. Let us first introduce some recent works on the numerical analysis of 2D stochastic incompressible Navier--Stokes equations. By utilizing a local monotonicity trick, \cite{MR3081484} deduces the convergence of finite-element-based space-time discretizations. Besides, a splitting up temporal semi-discretization, and an implicit and a semi-implicit space-time discretizations are investigated in \cite{MR3274888} and \cite{MR3022227}, respectively, where the mean square convergence rate localized on a sample subset with large probability and the speed of the convergence in probability are derived. Recently, Bessaih and Millet focus on temporal semi-discretizations in \cite{MR4019052} and a space-time fully discretization in \cite{2020arXiv200406932B}, and first give the results on the strong convergence speed. Their results reveal that the speed of the strong convergence depends on the diffusion coefficient and on the viscosity parameter. For other kinds of semilinear SEEs with non-global Lipschitz nonlinearity, their numerical analysis can be found in, for example, \cite{MR3813549,MR4019051,MR3826675,MR4123754} and references therein. For fully nonlinear SEEs, 
\cite{MR2139212} considers general monotone SEEs in variational framework and presents the convergence of various discretizations, although without rates of convergence. We would like mention that the nonlinear term in Euler equations is neither of Lipschitz nor of monotone type.

By introducing the so-called vorticity $\xi=\nabla^{\perp}\cdot u$, we study the vorticity equation (see \eqref{vorticity} below), which is equivalent to \eqref{Euler} in view of the Biot--Savart law (cf. \cite[Chapter 7]{MR2768550}).  Let $\left\{0=t_{0}<t_{1}<\cdots<t_{n}=T\right\}$ be a finite partition of $[0, T]$ with a uniform mesh $\tau=T/n$, $n\in\mathbb{N}^+$. We use the splitting technique to separate the noise and the nonlinear term, and propose the following splitting semi-implicit Euler (SIE) method for the vorticity equation 
\begin{align}
\bar{\xi}_{i+1}=\xi_{i}-\tau(K*\xi_{i})\cdot \nabla\bar{\xi}_{i+1},\quad \xi_{i+1}=\bar{\xi}_{i+1}+\Delta W^{curl}_{i+1}, \quad i=0,\dots,n-1, \label{method}
\end{align}
with $\Delta W^{curl}_{i+1}=W^{curl}(t_{i+1})-W^{curl}(t_{i})$ being the increment of $W^{curl}=\nabla^{\perp} \cdot W$. In comparison to the fact that the deterministic Euler flow preserves any $L^{p}(\mathbb{T}^2)$-norm, $1\le p\le \infty$, of the vorticity $\xi$, the SIE method is shown to guarantee the $L^{p}(\mathbb{T}^2)$-norm non-increasing. The unique solvability of \eqref{method} in the distribution sense is proved via a Galerkin approximation argument and Banach-Alaoglu theorem. Taking advantage of the integrability and divergence free property of the Biot–Savart kernel, we show that the numerical solutions $\{\xi_{i}\}_{i=0}^{n}$ belong to $W^{1,4}(\mathbb{T}^2)$ almost surely for sufficiently small step size, which implies that the unique weak solution of \eqref{method} is also a strong solution. 
We remark that without the dominant Laplacian operator, Euler equations behave far worse than Navier-Stokes equations, and the regularity estimates of the numerical soutions of the splitting SIE scheme for  Euler equations are quiet different from that for Navier-Stokes equations (see Remark \ref{regularity}). Besides, we show that the numerical solutions $\{\xi_i\}_{i=0}^n$ are bounded in $L^4(\mathbb{T}^2)$ almost surely with a bound proportional to $\tau^{-\varepsilon}$ for any $\varepsilon>0$, which will be used in the error estimates.

The present work measures the error between $\xi(t_{i})$ and $\xi_{i}$ in $L^{2}(\mathbb{T}^2)$, and aims to derive the optimal convergence order in probability. The rate of convergence in probability matches well with SEEs enjoying non-globally Lipschitz nonlinearities, whose study can be traced back to the work of Printems \cite{MR1873517}. The interaction between the stochastic forcing and the essentially quadratic nonlinear term $K*\xi\cdot \nabla \xi$ is one of main obstacles of numerical analysis for stochastic Euler equations. Our analysis of convergence in probability relies on two key ingredients, one being a priori estimate of pathwise convergence order, another being the application of the Lenglart-Rebolledo inequality.
Inspired by the work of Bessaih and Ferrario \cite{MR4107937}, we impose a suitable spatial regularity assumption on $W$ and demonstrate that almost all paths of the exact solution live in $W^{2,4}(\mathbb{T}^2)$ with the help of a Beale–Kato–Majda type inequality, the Moser estimate, and the commutator estimate. Based on the spatial regularity, we show that the temporal H\"older exponent of the exact solution in $H^1(\mathbb{T}^2)$ is close to $\frac{1}{2}$. By the orthogonality property $\langle K*\xi\cdot\nabla \zeta,\zeta\rangle=0$ of the nonlinear term, we deduce that the pathwise convergence order is no less than the H\"older exponent of the exact solution. Finally, applying a recursion argument, we improve the order of convergence in probability to nearly 1, which is optimal since our method only uses the information about the increments of $W$ (cf. \cite{MR609181}). To the best of our knowledge, this is the first work about the convergence analysis of numerical approximations for stochastic Euler equations on torus. 

The rest of this article is organized as follows. Some notations and basic inequalities are introduced in Section \ref{sec2}, which will be used throughout the work. In Section \ref{sec3}, we present the unique solvability of the splitting SIE method and regularity estimates of the associated numerical solutions. Section \ref{sec4} carries out the error analysis for the considered method in pathwise and probability senses.

\section{Preliminaries} \label{sec2}
This section is devoted to introducing some basic notations and mathematical tools used in the sequel. Throughout this paper, we use $C$ to denote a generic positive constant which is independent of step size $\tau$ and may differ from occurrence to occurrence. Sometimes, we write $C(\alpha, \beta)$ or $C_{\alpha,\beta}$ to emphasise its dependence on certain quantities $\alpha, \beta$.
\subsection{Mathematical setting}
For $n=1$ or $2$,  we denote by $\langle\cdot, \cdot \rangle$ the $[L^{2}(\mathbb{T}^2)]^n$-inner product  and by $|\cdot|_{p}$  the $\left[L^{p}(\mathbb{T}^2)\right]^{n}$-norm  for $p\geq 1$. For $m \geq 1$ and $p\geq 1$, $[W^{m, p}(\mathbb{T}^2)]^n$ and $[H^{m}(\mathbb{T}^2)]^n$ are the classical Sobolev spaces, endowed with the usual norms $\|\cdot\|_{m, p}$ and $\|\cdot\|_{m}$, respectively. Let $H_{1}$ and $H_{2}$ be two separable Hilbert spaces and $L_{2}(H_{1},H_{2})$ be the space of Hilbert-Schmidt operators from $H_{1}$ to $H_{2}$. Denote by $\mathcal{E}$ (resp. $\mathcal{H}$) the space of periodic functions (resp. periodic divergence free vector fields) which are square integrable with zero mean value on $\mathbb{T}^2$. Both $\mathcal{E}$ and $\mathcal{H}$ are separable Hilbert spaces, equipped with the $L^{2}(\mathbb{T}^2)$-inner product and $\left[L^{2}(\mathbb{T}^2)\right]^{2}$-inner product, respectively. Indeed, denoting $\mathbb{Z}_{0}^{2}:=\mathbb{Z}^{2} \backslash\{0\}$, $\mathbb{Z}_{+}^{2}:=\left\{k \in \mathbb{Z}_{0}^{2}: k_{1}>0 \text{ or } k_{1}=0, k_{2}>0\right\}~\text{and}~\mathbb{Z}_{-}^{2}=-\mathbb{Z}_{+}^{2}$, then $\left\{e_{k}\right\}_{k \in \mathbb{Z}_{0}^{2}}$ and $\left\{g_{k}\right\}_{k \in \mathbb{Z}_{0}^{2}}$  are complete orthonormal bases of $\mathcal{E}$ and $\mathcal{H}$, respectively (cf. \cite[Section 1]{MR2213944}):
\begin{align*}
e_{k}(x)=\left\{\begin{array}{ll}
\sqrt{2}\cos (2 \uppi k \cdot x), & k \in \mathbb{Z}_{+}^{2} ,\\
\sqrt{2}\sin (2 \uppi k \cdot x), & k \in \mathbb{Z}_{-}^{2},
\end{array}\right.
\end{align*}
and
\begin{align*}
g_{k}(x)=\frac{k^{\perp}}{|k|}e_{k}(x), \text{ with  }k^{\perp}=\left(k_{2},-k_{1}\right).
\end{align*}

Let $\mathcal{G}:=\{\nabla g:g\in H^{1}\left(\mathbb{T}^2\right) \cap\mathcal{E}\}$ be the space of gradients of periodic
functions in $H^{1}\left(\mathbb{T}^2\right)$. Then we have the following Helmholtz decomposition (see, e.g., \cite[p15]{MR0609732})
\begin{align}
[L^{2}(\mathbb{T}^2)]^2/\mathbb{R}^2=\mathcal{G}\oplus \mathcal{H}, \label{decomposition}
\end{align}
where $\oplus$ denotes the direct sum operation. The decomposition \eqref{decomposition} implies the existence of a unique orthogonal projection operator $\mathcal{P}:\left[L^{2}\left(\mathbb{T}^2\right)\right]^2/\mathbb{R}^2 \rightarrow \mathcal{H}$. 
Thanks to the incompressibility condition, $\bm{\pi}$ can be uniquely determined by $u$ through the following equation
\begin{align}
\nabla \bm{\pi} =(\mathcal{P}-{\rm Id})(u \cdot \nabla) u \label{pre}.
\end{align}
We formally write $\bm{\pi}=\nabla^{-1}(\mathcal{P}-{\rm Id})(u \cdot \nabla) u$ the unique solution of \eqref{pre} in $H^{1}(\mathbb{T}^2)\cap \mathcal{E}$.

Let $\left\{\beta_{k}(t) ; t \geq 0\right\}_{k \in \mathbb{Z}_{0}^{2}}$ be a sequence of independent standard real-valued Wiener processes defined on $(\Omega, \mathcal{F},\{\mathcal{F}_{t}\}_{t\geq 0}, \mathbb{P})$ and $\tilde{\beta}_k(\cdot):=-\beta_{-k}(\cdot),k \in \mathbb{Z}_{0}^{2}$.  The noise $W$ in Eq. (\ref{Euler}) is a Wiener process in $\mathcal{H}$ which has the following Karhunen-Loève expansion
\begin{align}
W(t)=\sum_{k \in \mathbb{Z}_{0}^{2}} c_{k} \beta_{k}(t) g_{k} \label{noise}
\end{align}
with $\{c_{k}\}_{k\in \mathbb{Z}_{0}^{2}} \subset \mathbb{R}_{+}$.

The following assumption on the spatial regularity of $W$ is required in the convergence analysis.
\begin{hyp} 
The sequence $\{c_{k}\}_{k\in \mathbb{Z}_{0}^{2}}$ in \eqref{noise} satisfies for some $h\ge\frac{9}{2}$  \label{assum}
\begin{align}
\sum_{k \in \mathbb{Z}_{0}^{2}} c_{k}^{2}\left\|g_{k}\right\|_{h}^{2}<\infty.  \label{condition}
\end{align}
\end{hyp}

For a divergence free velocity field $u$, we define the vorticity $\xi$ by $\xi=\nabla^{\perp} \cdot u=\partial_{1}u_{2}-\partial_{2}u_{1}$. Then the stochastic Euler equation (\ref{Euler}) is equivalent to the following vorticity equation (see e.g. \cite[Section 2]{MR2355417}):
\begin{align}
\mathrm{d}\xi(t)+(K *\xi(t))\cdot \nabla\xi(t) \mathrm{d} t=\mathrm{d} W^{curl}(t), \label{vorticity}
\end{align}
where $K$ is the Biot-Savart kernel and
\begin{align*}
W^{curl}(t)=\nabla^{\perp} \cdot W(t)=2\uppi \sum_{k \in \mathbb{Z}_{0}^{2}}  |k|c_{-k} \tilde{\beta}_{k}(t)  e_{k}.
\end{align*}
It is known that the divergence of $K*\xi$ vanishes and $\xi=\nabla^{\perp} \cdot (K*\xi)$. 

Under Assumption \ref{assum}, it can be seen that $W$ is an $[H^{h}(\mathbb{T}^2)]^{2}\cap \mathcal{H}$-valued $Q$-Wiener process with $Qg_k=c_k^2g_k, k\in\mathbb{Z}_{0}^{2}$, and $W^{curl}$ is an $H^{h-1}(\mathbb{T}^2)\cap \mathcal{E}$-valued $Q_{1}$-Wiener process with $Q_{1}e_{k}=(2\uppi|k|c_{-k})^{2}e_{k},k\in\mathbb{Z}_{0}^{2}$. It follows from the embedding $H^{h}(\mathbb{T}^2)\hookrightarrow  W^{m,\infty}(\mathbb{T}^2),m<h-1$, that the paths $W\in C\left(\mathbb{R}_{+} ; [W^{m, \infty}(\mathbb{T}^2)]^{2}\cap \mathcal{H}\right)$ a.s.

\subsection{Some useful inequalities}
The following lemma states that $K:L^p(\mathbb{T}^2)\to W^{1,p}(\mathbb{T}^2)$ is bounded (see, e.g., details in \cite[Theoren 2.2]{MR158189}). 
\begin{lem} \label{Kbound}
 For any $p \in[2, \infty)$ and $\xi\in\mathcal{E}\cap L^{p}(\mathbb{T}^2)$, we have 
 $$|K*\xi|_{p} \leq C  |\xi|_{p} $$ and
$$
 \frac{1}{C_p} |\xi|_{p}\leq |\nabla (K*\xi)|_{p} \leq C_p |\xi|_{p}. 
$$
\end{lem}
The embedding relation $W^{1,2}(\mathbb{T}^2) \hookrightarrow  L^{4}(\mathbb{T}^2)$ together with Lemma \ref{Kbound} yields
\begin{align}
|K*\xi|_{4}\leq C ||K*\xi||_{1,2}\leq C|\xi|_{2}. \label{K4}
\end{align}

The following Moser estimate and commutator estimate are useful tools in our analysis (cf. \cite[Section 2]{MR3161482}).
\begin{lem} 
Let $\mathcal{D}$ be a smooth bounded simply-connected domain in $\mathbb{R}^{d}$.

i) If $m>d/p$ and $f, g \in W^{m, p}(\mathcal{D})$, then there is a universal constant $C=C(m, p, \mathcal{D})>0$ such that
\begin{align}
\|f g\|_{m, p} \leq C\left(|f|_{\infty}\|g\|_{m, p}+|g|_{\infty}\|f\|_{m, p}\right). \label{Moser}
\end{align}
In particular, this shows that $W^{m, p}$ is an algebra whenever $m>d / p$. 

ii) If $m>1+d/p, f \in W^{m, p}(\mathcal{D})$, and $g \in$ $W^{m+1, p}(\mathcal{D})$, then there is a universal constant $C=C(m, p, \mathcal{D})>0$ such that
\begin{align}
\sum_{0 \leq|\alpha| \leq m}\left|\partial^{\alpha}(f \cdot \nabla g)-f \cdot \nabla \partial^{\alpha} g\right|_{p} \leq C\left(\|f\|_{m, p}|\nabla g|_{\infty}+|\nabla f|_{\infty}\|g\|_{m, p}\right). \label{commutator}
\end{align}
\end{lem}

The next lemma is useful to derive the order of convergence in probability (cf. \cite[\RNum{1}.9 Theorem 3]{MR1022664}). 
\begin{lem}\label{LRineq}
(Lenglart-Rebolledo inequality) Suppose that $M =\{M_t: t \geq 0\}$ is a real-valued continuous local martingale on $(\Omega, \mathcal{F},\{\mathcal{F}_{t}\}_{t\geq 0}, \mathbb{P})$ with $M_0=0$. Then, for all $\varepsilon>0$ and $\delta>0$
$$
\mathbb P\left(\sup _{0 \leq t \leq T}|M_t|>\varepsilon\right) \leq \frac{3\delta}{\varepsilon}+\mathbb P(\langle M\rangle_T^{\frac{1}{2}}>\delta),
$$
where $\langle M\rangle$ is the quadratic variation of $M$.
\end{lem}

To end this section, we introduce the Burkholder inequality in martingale type 2 Banach space $L^{q}(\mathbb{T}^2)$, $2\leq q<\infty$ (see \cite[Proposition 2.4]{MR2433958}). Denote by $R(\mathcal{E}, L^{q}(\mathbb{T}^2))$ the space of $\gamma$-radonifying operators from $\mathcal{E}$ into $L^{q}(\mathbb{T}^2)$, equipped with the norm
$$
\|\Psi\|_{R(\mathcal{E},  L^{q}(\mathbb{T}^2))}=\Big(\widetilde{\mathbb{E}}\Big|\sum_{k \in \mathbb{Z}_{0}^{2}} \gamma_{k} \Psi e_{k}\Big|_{q}^{2}\Big)^{\frac{1}{2}}, \quad \Psi\in R(\mathcal{E}, L^{q}(\mathbb{T}^2)),
$$
 where $\left\{\gamma_{k}\right\}_{k \in \mathbb{Z}_0^{2}}$ is a sequence of independent standard normal variables on another probability space $(\widetilde{\Omega}, \widetilde{\mathscr{F}},\widetilde{\mathbb{P}})$, and $\widetilde{\mathbb{E}}$ is the expectation with respect to $\widetilde{\mathbb{P}}$. For any $\{\mathcal{F}_t\}_{ t\in[0,T]}$-adapted process $\phi\in L^{p}\left(\Omega ; L^{2}([0, T] ; R(\mathcal{E}, L^{q}(\mathbb{T}^2)))\right)$ and some $C_{p,q}>0$,
\begin{align}
\left\|\sup _{t \in[0, T]}\left| \int_{0}^{t} \phi(r) \mathrm{d} W^{E}(r)\right|_{q}\right\|_{L^{p}(\Omega)}^p  \leq C_{p, q}\mathbb{E}\left(\int_{0}^{T}\|\phi(t)\|_{R(\mathcal{E},  L^{q}(\mathbb{T}^2))}^{2} \mathrm{~d} t\right)^{\frac{p}{2}}, \label{BDG}
\end{align}
where the process $W^{E}:=\sum_{k \in \mathbb{Z}_{0}^{2}} \tilde{\beta}_{k} e_{k}$ is a cylindrical Wiener process on $\mathcal{E}$.
Moreover, the norm $\|\cdot\|_{R(\mathcal{E},  L^{q}(\mathbb{T}^2))}$ has the following estimate (see, e.g., \cite[Lemma 2.1]{MR2433958}): for some $C_{q}>0$,
\begin{align}
\|\Psi\|_{R(\mathcal{E},  L^{q}(\mathbb{T}^2))}^{2} \leqslant C_{q}\Big|\sum_{k \in \mathbb{Z}_{0}^{2}}\left(\Psi e_{k}\right)^{2}\Big|_{\frac{q}{2}}, \quad \Psi \in R(\mathcal{E},  L^{q}(\mathbb{T}^2)). \label{norm}
\end{align}

\section{Splitting semi-implicit Euler method} \label{sec3}
In this section, we investigate the unique solvability for the splitting SIE method \eqref{method} and the regularity estimates for the associated numerical solutions. Throughout this paper, we suppose that the initial datum $\xi(0)$ is $\mathcal F_0$-adapted. 
For fixed $\xi\in \mathcal{E}$ and $\tau>0$, let $\Phi_{\tau}(\xi)$ be determined by
\begin{align} 
\Phi_{\tau}(\xi)=\xi-\tau(K*\xi)\cdot \nabla \Phi_{\tau}(\xi). \label{Phi}
\end{align}
By \eqref{Phi}, the numerical solution of the splitting SIE method \eqref{method} can be rewritten as 
\begin{align}
\xi_{0}=\xi(0), \ \xi_{i+1}=\Phi_{\tau}(\xi_{i})+\Delta W^{curl}_{i+1}, \ 0\leq i\le n-1. \label{numericalsol}
\end{align}
Then we obtain an $\{\mathcal{F}_{t_i}\}_{i=0}^n$-adapted sequence $\{\xi_i\}_{i=0}^n$. Formally, $u_{i}=K*\xi_{i}$ and $\bm \pi_{i}=\nabla^{-1}(\mathcal{P}-{\rm Id})(u_{i} \cdot \nabla) u_{i}$ are approximations of $u(t_{i})$ and $\bm\pi(t_{i})$, respectively. In the
sequel, denote $\kappa_n(t):=t_{i}$ and $\eta_n(t):=t_{i+1}$ if $t\in[t_{i},t_{i+1})$, $i=0,1,\ldots,n-1$.

We first present the existence of a weak solution of \eqref{Phi} by a Galerkin approximation argument. In the proof we will need the following lemma which is an easy consequence of the Brouwer fixed point theorem but a useful tool to prove the existence of a solution for a finite dimensional system. 
\begin{lem}\cite[Lemma \RNum{2}. 1.4]{MR0609732}\label{finitefix}
Let $X$ be a finite dimensional Hilbert space with scalar product $\langle\cdot, \cdot\rangle_X$ and norm $\|\cdot\|_X$, and let $P$ be a continuous mapping from $X$ into itself such that
$$
\langle P(x), x\rangle_X>0,\ \forall \ \|x\|_X=c,
$$
for some $c>0$. Then there exists $x \in X,\|x\|_X \leq c,$ such that $ P(x)=0.$
\end{lem}
\begin{tho} \label{existence}
For fixed $\xi\in \mathcal{E}$ and $\tau>0$, there is a solution $\Phi_{\tau}(\xi)\in \mathcal{E}$ of Eq. \eqref{Phi} in the sense that
$$
\langle\Phi_{\tau}(\xi),\zeta \rangle=\langle\xi,\zeta \rangle+\tau\langle (K*\xi)\cdot \nabla\zeta,\Phi_{\tau}(\xi) \rangle, \quad \forall \ \zeta\in W^{1,4}(\mathbb{T}^2)\cap \mathcal{E}.
$$
\end{tho}
\begin{proof}
Let $\Lambda_{N}:=\{k\in\mathbb{Z}_{0}^{2}:|k_{1}|\vee |k_{2}|\leq N\}$ be a finite subset of $\mathbb{Z}_{0}^{2}$ and $\mathcal{E}_{\Lambda_{N}}=\operatorname{span}\left\{e_{k}: k \in \Lambda_{N}\right\}$. Define the projection operator $\Pi_{\Lambda_{N}}: \mathcal{E} \rightarrow \mathcal{E}_{\Lambda_{N}}$ as
$$
\Pi_{\Lambda_{N}} \zeta =\sum_{k \in \Lambda_{N}}\left\langle \zeta, e_{k}\right\rangle e_{k}, \quad \zeta  \in \mathcal{E}.
$$
Now we consider the finite dimensional approximation of Eq. (\ref{Phi}):
\begin{align} 
\Phi_{\tau}^{N}(\xi)=\Pi_{\Lambda_{N}}\xi-\tau\Pi_{\Lambda_{N}}[(K*\Pi_{\Lambda_{N}}\xi)\cdot \nabla\Phi_{\tau}^{N}(\xi)].  \label{finitedim}
\end{align} 
Let $P:\mathcal{E}_{\Lambda_{N}}\to\mathcal{E}_{\Lambda_{N}}$ be defined by
\begin{align*} 
P(x)=-\Pi_{\Lambda_{N}}\xi+\tau\Pi_{\Lambda_{N}}[(K*\Pi_{\Lambda_{N}}\xi)\cdot \nabla x]+x, \ \forall \ x\in \mathcal{E}_{\Lambda_{N}}.
\end{align*} 
Obviously $P$ is a linear map from $\mathcal{E}_{\Lambda_{N}}$ to $\mathcal{E}_{\Lambda_{N}}$, thus it's continuous. For any $c> |\Pi_{\Lambda_{N}}\xi|_2\ge0$, we have
\begin{align*}
\langle P(x),x\rangle=\frac{1}{2}(|x|_2^2-|\Pi_{\Lambda_{N}}\xi|_2^2+|x-\Pi_{\Lambda_{N}}\xi|_2^2)>0, \ \forall\ |x|_2=c.
\end{align*}
By Lemma \ref{finitefix}, we know that there exists a $\Phi_{\tau}^N(\xi)\in \{x\in\mathcal{E}_{\Lambda_{N}}: |x|_2\le c\}$ such that $\Phi_{\tau}^N(\xi)$ satisfies \eqref{finitedim}.
If $\Psi_{\tau}^N(\xi)\in \mathcal{E}_{\Lambda_{N}}$ is another solution of \eqref{finitedim}, then
\begin{align*}
|\Phi_{\tau}^N(\xi)-\Psi_{\tau}^N(\xi)|_2^2=\tau\langle K*\xi\cdot \nabla(\Phi_{\tau}^N(\xi)-\Psi_{\tau}^N(\xi)),\Phi_{\tau}^N(\xi)-\Psi_{\tau}^N(\xi)\rangle=0.
\end{align*}
Therefore, Eq. (\ref{finitedim}) has a unique solution $\Phi_{\tau}^{N}(\xi)\in \mathcal{E}_{\Lambda_{N}}$. Notice that
\begin{align*}
|\Phi_{\tau}^{N}(\xi)|_{2}^{2}&=\langle\Pi_{\Lambda_{N}}\xi,\Phi_{\tau}^{N}(\xi)\rangle-\langle\tau\Pi_{\Lambda_{N}}[(K*\Pi_{\Lambda_{N}}\xi)\cdot \nabla\Phi_{\tau}^{N}(\xi)],\Phi_{\tau}^{N}(\xi)
\rangle\\
&=\langle\Pi_{\Lambda_{N}}\xi,\Phi_{\tau}^{N}(\xi)\rangle-\tau \langle (K*\Pi_{\Lambda_{N}}\xi)\cdot \nabla\Phi_{\tau}^{N}(\xi),\Pi_{\Lambda_{N}}\Phi_{\tau}^{N}(\xi)\rangle\\
&=\langle\Pi_{\Lambda_{N}}\xi,\Phi_{\tau}^{N}(\xi)\rangle\\
&\leq |\Pi_{\Lambda_{N}}\xi|_{2}|\Phi_{\tau}^{N}(\xi)|_{2},
\end{align*}
which implies
\begin{align*}
|\Phi_{\tau}^{N}(\xi)|_{2}\leq |\Pi_{\Lambda_{N}}\xi|_{2}\leq |\xi|_{2}.
\end{align*}
By Banach-Alaoglu theorem, we obtain that there is a subsequence $\{\Phi_{\tau}^{N_{l}}(\xi)\}_{l\in \mathbb{N}}$ of $\{\Phi_{\tau}^{N}(\xi)\}_{N\in \mathbb{N}}$ and a $\Phi_{\tau}(\xi)\in \mathcal{E}$ such that for any $\zeta \in \mathcal{E}$,
\begin{align*}
\langle\Phi_{\tau}^{N_{l}}(\xi),\zeta \rangle \to \langle\Phi_{\tau}(\xi),\zeta \rangle \text{ as } l\to \infty.
\end{align*}
On the other hand, for any $\zeta\in W^{1,4}(\mathbb{T}^2)\cap \mathcal{E}$,
\begin{align*}
\langle\Phi_{\tau}^{N_{l}}(\xi),\zeta \rangle &=\langle \Pi_{\Lambda_{N_{l}}}\xi,\zeta \rangle+\tau\langle (K*\Pi_{\Lambda_{N_{l}}}\xi)\cdot \nabla\Pi_{\Lambda_{N_{l}}}\zeta,\Phi_{\tau}^{N_{l}}(\xi) \rangle\\
&\to \langle\xi,\zeta \rangle+\tau\langle (K*\xi)\cdot \nabla\zeta,\Phi_{\tau}(\xi) \rangle,\text{ as } l\to \infty.
\end{align*}
Hence, there is a weak solution $\Phi_{\tau}(\xi)\in \mathcal{E}$ of Eq. \eqref{Phi}.
\end{proof}

In preparation for proving the existence and uniqueness of strong solution for \eqref{Phi}, we give the following a priori estimate, which shows that $\Phi_{\tau}:L^{p}(\mathbb{T}^2)\to L^{p}(\mathbb{T}^2)$ is bounded.
\begin{lem} \label{priest}
Let $1\le p\le\infty$ and $\xi \in L^{p}(\mathbb{T}^2)\cap \mathcal{E}$. If Eq. (\ref{Phi}) has a strong solution $\Phi_{\tau}(\xi)$ for some fixed $\tau>0$, then
$$|\Phi_{\tau}(\xi)|_{p}\leq |\xi|_{p}.$$ 
\end{lem}
\begin{proof}
We first treat the case $1\le p<\infty$. By multiplying $ |\Phi_{\tau}(\xi)|^{p-1}\mathrm{sign}(\Phi_{\tau}(\xi))$ on both sides of Eq. \eqref{Phi} and then integrating over $\mathbb{T}^2$, we obtain
\begin{align*}
|\Phi_{\tau}(\xi)|_{p}^{p}=&\int_{\mathbb{T}^2}\mathrm{sign}(\Phi_{\tau}(\xi)) |\Phi_{\tau}(\xi)|^{p-1}[\xi-\tau(K*\xi) \cdot \nabla\Phi_{\tau}(\xi)]\mathrm{d}x\\
=&\int_{\mathbb{T}^2}\mathrm{sign}(\Phi_{\tau}(\xi))\xi|\Phi_{\tau}(\xi)|^{p-1} \mathrm{d}x-\frac{\tau}{p}\int_{\mathbb{T}^2}(K*\xi) \cdot \nabla|\Phi_{\tau}(\xi)|^{p}\mathrm{d}x\\
=&\int_{\mathbb{T}^2}\mathrm{sign}(\Phi_{\tau}(\xi))\xi|\Phi_{\tau}(\xi)|^{p-1} \mathrm{d}x+\frac{\tau}{p}\int_{\mathbb{T}^2}[\nabla\cdot(K*\xi)] |\Phi_{\tau}(\xi)|^{p}\mathrm{d}x\\
=&\int_{\mathbb{T}^2}\mathrm{sign}(\Phi_{\tau}(\xi))\xi|\Phi_{\tau}(\xi)|^{p-1} \mathrm{d}x,
\end{align*}
where we have used the integration by parts and the fact that $ K*\xi$ is divergence free. Using H\"older inequality, we have
\begin{align*}
|\Phi_{\tau}(\xi)|_{p}\leq |\xi|_{p}.
\end{align*}
Letting $p\to\infty$, we complete the proof.
\end{proof}

The next lemma gives the existence and uniqueness of strong solution for \eqref{Phi}.
\begin{lem}
For $\xi\in W^{1,4}(\mathbb{T}^2)\cap \mathcal{E}$ and sufficiently small step size $\tau>0$, Eq. \eqref{Phi} has a strong unique solution $\Phi_{\tau}(\xi)\in W^{1,4}(\mathbb{T}^2)\cap \mathcal{E}$. 
\end{lem}
\begin{proof}
{\bf{Uniqueness:}} If $\Psi_{\tau}(\xi)\in W^{1,4}(\mathbb{T}^2)\cap \mathcal{E}$ is another solution of \eqref{Phi}, then
\begin{align*}
|\Phi_{\tau}(\xi)-\Psi_{\tau}(\xi)|_2^2=\tau\langle K*\xi\cdot \nabla(\Phi_{\tau}(\xi)-\Psi_{\tau}(\xi)),\Phi_{\tau}(\xi)-\Psi_{\tau}(\xi)\rangle=0,
\end{align*}
and the uniqueness follows.

{\bf{Existence:}} Taking the partial derivative $\partial_{i}$ of (\ref{Phi}),
\begin{align} 
\partial_{i}\Phi_{\tau}(\xi)=\partial_{i}\xi-\tau\partial_{i}[(K*\xi)\cdot \nabla\Phi_{\tau}(\xi)], \quad i=1,2. \label{aaa}
\end{align} 
Multiplying \eqref{aaa} by $\partial_{i}\Phi_{\tau}(\xi)|\nabla \Phi_{\tau}(\xi)|^{2}$, summing over $i$, and  integrating over $\mathbb{T}^2$, we obtain
\begin{align*} 
|\nabla\Phi_{\tau}(\xi)|_{4}^{4}=&\sum_{i=1}^{2}\langle\partial_{i}\xi-\tau\partial_{i}[(K*\xi)\cdot \nabla\Phi_{\tau}(\xi)],\partial_{i}\Phi_{\tau}(\xi)|\nabla \Phi_{\tau}(\xi)|^{2}\rangle\\
\le&\frac{1}{4}|\nabla\xi|_{4}^{4} +\frac{3}{4}|\nabla \Phi_{\tau}(\xi)|_{4}^{4}-\tau\sum_{i=1}^{2}\langle\partial_{i}[(K*\xi)\cdot \nabla\Phi_{\tau}(\xi)],\partial_{i}\Phi_{\tau}(\xi)|\nabla \Phi_{\tau}(\xi)|^{2}\rangle,
\end{align*} 
where the sum is further decomposed into
\begin{align*} 
&\sum_{i}\langle\partial_{i}[(K*\xi)\cdot \nabla\Phi_{\tau}(\xi)],\partial_{i}\Phi_{\tau}(\xi)|\nabla \Phi_{\tau}(\xi)|^{2}\rangle\\
=&\sum_{i,j}\langle\partial_{i}(K*\xi)_{j}\partial_{j}\Phi_{\tau}(\xi),\partial_{i}\Phi_{\tau}(\xi)|\nabla \Phi_{\tau}(\xi)|^{2}\rangle\\
&+\sum_{i,j}\langle(K*\xi)_{j}\partial_{ij}\Phi_{\tau}(\xi),\partial_{i}\Phi_{\tau}(\xi)|\nabla \Phi_{\tau}(\xi)|^{2}\rangle\\
=&:A+B.
\end{align*} 
Note that for all $p \in[1,2), K \in L^{p}\left(\mathbb{T}^2\right)$ (see e.g. \cite{MR3947288}). By Young inequality for convolution,
\begin{align*} 
|A|\leq |\nabla(K*\xi)|_{\infty}|\nabla \Phi_{\tau}(\xi)|_{4}^{4}\leq C\|\xi\|_{1,4}|\nabla \Phi_{\tau}(\xi)|_{4}^{4}.
\end{align*} 
Using the integration by parts and the fact that $K*\xi$ is divergence free, we see that
\begin{align*} 
B=&\sum_{i,j}\int_{\mathbb{T}^2}(K*\xi)_{j}\partial_{ij}\Phi_{\tau}(\xi)\partial_{i}\Phi_{\tau}(\xi)|\nabla \Phi_{\tau}(\xi)|^{2}\ud x\\
=&-\sum_{i,j}\int_{\mathbb{T}^2}\partial_{j}(K*\xi)_{j}\partial_{i}\Phi_{\tau}(\xi)|\nabla \Phi_{\tau}(\xi)|^{2}\partial_{i}\Phi_{\tau}(\xi)\ud x\\
&-\sum_{i,j}\int_{\mathbb{T}^2}(K*\xi)_{j}\partial_{ij}\Phi_{\tau}(\xi)|\nabla \Phi_{\tau}(\xi)|^{2}\partial_{i}\Phi_{\tau}(\xi)\ud x\\
&-\sum_{i,j}\int_{\mathbb{T}^2}(K*\xi)_{j}\partial_{i}\Phi_{\tau}(\xi)\partial_{j}|\nabla \Phi_{\tau}(\xi)|^{2}\partial_{i}\Phi_{\tau}(\xi)\ud x\\
=&-B-\frac{1}{2}\sum_{j}\int_{\mathbb{T}^2}(K*\xi)_{j}\partial_{j}|\nabla \Phi_{\tau}(\xi)|^{4}\ud x\\
=&-B,
\end{align*} 
and therefore $B=0$. 
To summarize, we have shown
\begin{align} 
|\nabla\Phi_{\tau}(\xi)|_{4}^{4}\leq |\nabla\xi|_{4}^{4} +4\tau C\|\xi\|_{1,4}|\nabla \Phi_{\tau}(\xi)|_{4}^{4}. \label{gradientest}
\end{align} 
Combining Lemma \ref{priest} and (\ref{gradientest}), we have for sufficiently small step size $\tau>0$, 
\begin{align*} 
\|\Phi_{\tau}(\xi)\|_{1,4}\leq C(\tau,\xi)\|\xi\|_{1,4},
\end{align*} 
which together with Theorem \ref {existence}, indicates that $\Phi_{\tau}(\xi)\in W^{1,4}(\mathbb{T}^2)\cap \mathcal{E}$ is the unique strong solution of Eq. \eqref{Phi}.
\end{proof}

\begin{rem} \label{regularity}
Applying the SIE method to deterministic Naiver-Stokes equations yields
\begin{align*}
\Phi_{\tau}(\xi)=\xi+\tau\nu\Delta \Phi_{\tau}(\xi)-\tau(K*\xi)\cdot \nabla\Phi_{\tau}(\xi) ,\ \nu>0.
\end{align*}
 The integration by parts gives
\begin{align*}
|\Phi_{\tau}(\xi)|_2^2+2\tau\nu|\nabla \Phi_{\tau}(\xi)|_2^2\le |\xi|_2^2,
\end{align*}
and therefore 
\begin{align*}
\|\Phi_{\tau}(\xi)\|_1\le \frac{1}{ \min\{1, \sqrt{2\tau\nu } \}}|\xi|_2.
\end{align*}
Compared with \eqref{gradientest}, this estimate shows that the numerical analysis of Euler equations is different from that of Navier-Stokes equations due to the absence of Laplacian.
\end{rem}

The next proposition gives the uniform boundedness of the numerical solutions in $L^{p}(\Omega,L^{p}(\mathbb{T}^2))$ via the Burkholder inequality in martingale type 2 Banach spaces. Furthermore, the numerical solutions $\{\xi_i\}_{i=0}^n$ are shown to be bounded in $L^4(\mathbb{T}^2)$ a.s., and the bound is proportional to $\tau^{-\varepsilon}$ for any $\varepsilon>0$.
\begin{pro}\label{lm3.2}
Suppose that  $\xi(0)\in L^{p}(\Omega,L^{p}(\mathbb{T}^2))\cap\mathcal E$ for some even number $p\geq 2$ and  
$$\sum_{k \in \mathbb{Z}^{2}_{0}}|k|^{2}c_{-k}^{2}\left|e_{k}\right|_{p}^{2}<\infty.$$
Then for all $1\le q\le p$,
\begin{align}
\sup_{1\leq i\leq n}||\xi_{i}||_{L^{q}(\Omega,L^{q}(\mathbb{T}^2))}\leq C(p,T,\xi(0)). \label{unib}
\end{align}   
Morever, if $\xi(0)\in L^{p}(\Omega,L^{p}(\mathbb{T}^2))$ for all large enough $p\in\mathbb{N}$, then we have 
\begin{align}
\mathbb{E}\sup_{1\leq i\leq n}|\xi_{i}|_4^p\leq C(p,T,\xi(0)), \label{Esup}
\end{align}
and for any $\varepsilon>0$,
\begin{align*}
\sup_{1\leq i\leq n}|\xi_{i}|_4\leq C(T,\xi(0),\omega,\varepsilon)\tau^{-\varepsilon}, \ a.s.
\end{align*}
\end{pro}
\begin{proof}  
We only need to prove \eqref{unib} for $q=p$ since the case $1\le q<p$ follows from the Hölder inequality. By binomial theorem, Young inequality, and \eqref{numericalsol}, we obtain
\begin{align*}
|\xi_{i+1}|^{p}&\leq |\Phi_{\tau}(\xi_{i})|^{p}+p\Phi_{\tau}(\xi_{i})^{p-1}\Delta W^{curl}_{i+1}+C_{p}[\Phi_{\tau}(\xi_{i})^{p-2}|\Delta W^{curl}_{i+1}|^{2}+|\Delta W^{curl}_{i+1}|^{p}].
\end{align*} 
The stochastic Fubini theorem gives
\begin{align*}
\mathbb{E}|\xi_{i+1}|_{p}^{p}&\leq \mathbb{E}|\Phi_{\tau}(\xi_{i})|_{p}^{p}+C_{p}\mathbb{E}\int_{\mathbb{T}^2}\Phi_{\tau}(\xi_{i})^{p-2}|\Delta W^{curl}_{i+1}|^{2}\ud x+C_{p}\mathbb{E}|\Delta W^{curl}_{i+1}|_{p}^{p}\\
&\leq \mathbb{E}|\xi_{i}|_{p}^{p}+C_{p}\left(1+\mathbb{E}|\xi_{i}|_{p}^{p}\right)\mathbb{E}|\Delta W^{curl}_{i+1}|_{p}^{2}+C_{p}\mathbb{E}|\Delta W^{curl}_{i+1}|_{p}^{p}.
\end{align*}
Recalling that $W^{E}=\sum_{k \in \mathbb{Z}_{0}^{2}} \tilde{\beta}_{k} e_{k}$ is an $\mathcal{E}$-cylindrical Wiener process, then
\begin{align*}
\Delta W^{curl}_{i+1}=\int_{t_{i}}^{t_{i+1}}\sum_{k \in \mathbb{Z}_{0}^{2}} 2\uppi |k|c_{-k} \ud\tilde{\beta}_{k}(t)  e_{k}=:\int_{t_{i}}^{t_{i+1}}\phi(t)\ud W^{E}(t),
\end{align*}
where $\phi(t)e_{k}=2\uppi |k|c_{-k}e_{k}$, $t\in[t_{i},t_{i+1}]$. It follows from \eqref{BDG} and \eqref{norm} that
\begin{align}
\mathbb{E}|\Delta W^{curl}_{i+1}|_{p}^{p}
\le C_{p}\left(\sum_{k \in \mathbb{Z}^{2}_{0}}|k|^{2}c_{-k}^{2}\left|e_{k}\right|_{p}^{2}\right)^{\frac{p}{2}}\tau^{\frac{p}{2}}\leq C_p\tau^{\frac{p}{2}}. \label{wcurlmoment}
\end{align}
Since $p\geq 2$, using H\"older inequality, we have
\begin{align*}
\mathbb{E}|\Delta W^{curl}_{i+1}|_{p}^{2}\leq\left(\mathbb{E}|\Delta W^{curl}_{i+1}|_{p}^{p}\right)^{\frac{2}{p}}
\leq C_p\tau.
\end{align*}
In conclusion, we have
\begin{align*}
\mathbb{E}|\xi_{i+1}|_{p}^{p}\leq (1+C_p\tau)\mathbb{E}|\xi_{i}|_{p}^{p}+C_p \tau,
\end{align*}
and \eqref{unib} follows from Gronwall lemma.

For \eqref{Esup}, it is sufficient to prove the case $p=4q$, for any $q \geq 1$. By Lemma \ref{priest}, we have
\begin{align*}
|\xi_{i+1}|_4^{4}
&\le |\xi_{i}|_4^{4}+4\langle\Phi_{\tau}(\xi_{i})^{3},\Delta W^{curl}_{i+1}\rangle+C\left[|\Phi_{\tau}(\xi_{i})|_4^{2}|\Delta W^{curl}_{i+1}|_4^{2}+|\Delta W^{curl}_{i+1}|_4^{4}\right]\\
&\le  |\xi_{0}|_4^{4}+4\sum_{l=0}^{i}\langle\Phi_{\tau}(\xi_{l})^{3},\Delta W^{curl}_{l+1}\rangle+C\sum_{l=0}^{i}\left[|\xi_{l}|_4^{2}|\Delta W^{curl}_{l+1}|_4^{2}+|\Delta W^{curl}_{l+1}|_4^{4}\right],
\end{align*} 
which implies
\begin{align*}
\mathbb{E}\sup_{1\le i\le n}|\xi_{i}|_4^{p}\le C_q\left(\mathbb{E} |\xi_{0}|_4^{p}+\RNum{1}+\RNum{2}+\RNum{3}\right)
\end{align*} 
with
\begin{align*}
\RNum{1}:=&\mathbb{E}\sup_{0\le t\le T}\left|\int_0^{t}\langle\Phi_{\tau}(\xi_{\kappa_n(s)})^{3},\ud W^{curl}(s)\rangle\right|^q,\\
\RNum{2}:=&\mathbb{E}\left(\sum_{l=0}^{n}|\xi_{l}|_4^{2}|\Delta W^{curl}_{l+1}|_4^{2}\right)^q,\\
\RNum{3}:=&\mathbb{E}\left(\sum_{l=0}^{n}|\Delta W^{curl}_{l+1}|_4^{4}\right)^q.
\end{align*} 
The Burkholder-Davis-Gundy inequality, \eqref{unib}, and the Minkowski inequality indicate
\begin{align*}
\RNum{1}\le&\mathbb{E}\left|\sum_{k\in\mathbb{Z}_0^2}\int_0^{T}\langle\Phi_{\tau}(\xi_{\kappa_n(s)})^{3},Q_1^{\frac{1}{2}}e_k\rangle^2\ud s\right|^\frac{q}{2}\\
\le&\left(\sum_{k\in\mathbb{Z}_0^2}\int_0^{T}\left\||\Phi_{\tau}(\xi_{\kappa_n(s)})|_6^6\right\|_{L^{\frac{q}{2}}(\Omega)}|Q_1^{\frac{1}{2}}e_k|_2^2\ud s\right)^\frac{q}{2}\\
\le&C\left(\sum_{k\in\mathbb{Z}_0^2}c_{-k}^2|k|^2\sup_{0\le i\le n}\left\|\xi_{i}\right\|_{L^{3q}(\Omega,L^6(\mathbb{T}^2))}^6\right)^\frac{q}{2}\\
\le&C(p,T,\xi(0)).
\end{align*}
By \eqref{unib} and \eqref{wcurlmoment}, we obtain
\begin{align*}
\RNum{2}\le\left(\sum_{l=0}^{n}\left\||\xi_{l}|_4^{2}|\Delta W^{curl}_{l+1}|_4^{2}\right\|_{L^q(\Omega)}\right)^q\le C_p\left(\sum_{l=0}^{n}(\mathbb{E}|\xi_{l}|_4^{2q})^{\frac{1}{q}}\tau\right)^q\le C(p,T,\xi(0)),
\end{align*}
and
\begin{align*}
\RNum{3}\le \left(\sum_{l=0}^{n}\left\||\Delta W^{curl}_{l+1}|_4^{4}\right\|_{L^p(\Omega)}\right)^q\le C(p,T).
\end{align*}
Hence, \eqref{Esup} follows from the above estimates. 

Let $X_n=\sup_{1\le i\le n}|\xi_{i}|_4$ and $Z_n=X_n/n$. Then
$$
\left(\mathbb{E}\left[\left|Z_{n}\right|^{p}\right]\right)^{1 / p} \leq C(p,T,\xi(0))  n^{-1}
$$
for all $p \in \mathbb{N}$ and all $n \in \mathbb{N}$. It follows from Chebyshev--Markov inequality and Borel--Cantelli lemma that for all $\varepsilon>0$ and  all $n \in \mathbb{N}$
$$
\left|Z_{n}(\omega)\right| \leq C(T,\xi(0),\varepsilon) n^{-1+\varepsilon} \quad \text { a.s.,}
$$
which implies
$$
X_{n}(\omega)\leq C(T,\xi(0),\varepsilon)  \tau^{-\varepsilon} \quad \text { a.s. }
$$
The proof is completed.
\end{proof} 

\section{Convergence order of numerical method} \label{sec4}
In this section, we are devoted to the convergence analysis of the splitting SIE method, including the pathwise convergence rate and the rate of convergence in probability. Unless necessary, we omit the notation ‘$\mathbb{P}$-a.s.’ or ‘a.s.’
\subsection{Pathwise convergence analysis}
If the initial datum $\xi(0)\in W^{1,4}(\mathbb{T}^2)$ and \eqref{condition} holds for $h>4$, \cite{MR4107937} proves that the exact solution $\xi\in L^{\infty}\left(0, T; W^{1,4}(\mathbb{T}^2)\right) \cap C_{w}\left(\left[0, T\right] ; W^{1,4}(\mathbb{T}^2)\right)$ a.s. For the purpose of obtaining the convergence rate, we begin with deriving the $W^{2,4}(\mathbb{T}^2)$ regularity of the exact solution. Throughout this part, we mainly use the pathwise argument.
\begin{pro}\label{W24}
Let Assumption \ref{assum} hold, and  
$\xi(0)\in H^{m}(\mathbb{T}^2) \cap\mathcal E$ almost surely for some $m\ge\frac{5}{2}$. Then $\xi \in L^{\infty}(0, T; W^{2,4}(\mathbb{T}^2))$, 
$\mathbb{P}$-a.s. Moreover, 
\begin{align*}
\sup _{0\leq t \leq T}|| \xi(t)||_{2,4} \leq C\left(T,\xi(0),\omega \right), \  a.s.
\end{align*} 
\end{pro}
\begin{proof}
By the regularity of $\xi(0)$, it holds that $\xi(0) \in W^{2,4}(\mathbb{T}^2)\cap \mathcal{E}$ since $m\ge \frac{5}{2}$. From \cite[Theorem 5]{MR4107937}, we know that $\xi \in L^{\infty}(0, T ; W^{1,4}(\mathbb{T}^2))$, $\mathbb{P}$-a.s. Thus, it remains to  estimate $\nabla^{2} \xi$ and look for $|\nabla^{2} \xi| \in L^{\infty}\left(0, T ; L^{4}(\mathbb{T}^2)\right)$, a.s.  

Let us take the Hessian for Eq. (\ref{vorticity}):
$$
\ud \nabla^{2} \xi+\nabla^{2}(u \cdot \nabla \xi) \ud t=\ud \nabla^{2} W^{c u r l},
$$
that can be rewritten for each component of the Hessian as
$$
\ud \partial_{ij} \xi+\partial_{ij}(u \cdot \nabla \xi) \ud t=\ud \partial_{ij} W^{c u r l}, \quad i,j=1,2.
$$
By defining $\eta=\xi-W^{c u r l}$, we have
$$
\frac{\partial}{\partial t} \partial_{ij} \eta+\partial_{ij}[u \cdot \nabla \eta]=-\partial_{ij}[u \cdot \nabla W^{c u r l}], \quad i,j=1,2.
$$
Let us multiply the above equation by $\partial_{ij} \eta|\nabla^{2} \eta|^{2}$, sum over $i,j$, and then integrate over $\mathbb{T}^2$. Then it holds that
\begin{align}\label{eta2}
\frac{1}{4} \frac{\ud}{\ud t}|\nabla^{2} \eta(t)|_{4}^{4}=&-\sum_{i,j=1}^{2}\left\langle\partial_{ij}[u \cdot \nabla \eta], \partial_{ij} \eta|\nabla^{2} \eta|^{2}\right\rangle\notag\\
&-\sum_{i,j=1}^{2}\left\langle\partial_{ij}[u \cdot \nabla W^{c u r l}], \partial_{ij} \eta|\nabla^{2} \eta|^{2}\right\rangle,
\end{align}
where the second sum is bounded as
\begin{align*}
\Big| \sum_{i,j}\langle\partial_{ij}[u \cdot \nabla W^{c u r l}], \partial_{ij} \eta|\nabla^{2} \eta|^{2}\rangle\Big|  & \leq C|| \nabla W^{curl}||_{2,4}|\nabla^{2}u|_{\infty}|\nabla^{2}\eta|_{4}^{3}\\
&\leq |\nabla^{2}\eta|_{4}^{4} +C|| \nabla W^{curl}||_{2,4}^{4}|\nabla^{2}u|_{\infty}^{4}.
\end{align*}
For the first sum of the right hand of \eqref{eta2}, a direct application of  chain rule  gives
\begin{align*}
&\sum_{i,j}\left\langle\partial_{ij}[u \cdot \nabla \eta], \partial_{ij} \eta|\nabla^{2} \eta|^{2}\right\rangle\\
=&\sum_{i,j,k}\left\langle\partial_{ij}u_{k} \partial_{k} \eta, \partial_{ij} \eta|\nabla^{2} \eta|^{2}\right\rangle+2\sum_{i,j,k}\left\langle\partial_{j}u_{k} \partial_{ik} \eta, \partial_{ij} \eta|\nabla^{2} \eta|^{2}\right\rangle\\
&+\sum_{i,j,k}\left\langle u_{k} \partial_{ikj} \eta, \partial_{ij} \eta|\nabla^{2} \eta|^{2}\right\rangle\\
=&:\RNum{1}+2\RNum{2}+\RNum{3}.
\end{align*}
It follows from Young and H\"older's inequalities that
\begin{align*}
|\RNum{1}|
\leq C\int_{\mathbb{T}^2}|\nabla^{2}u||\nabla\eta||\nabla^{2}\eta|^{3}\ud x\leq C|\nabla^{2} u|_{\infty}|\nabla\eta|_{4}
|\nabla^{2}\eta|_{4}^{3}\le |\nabla^{2}\eta|_{4}^{4}+C|\nabla^{2} u|^4_{\infty}|\nabla\eta|^4_{4}
\end{align*}
and
\begin{align*}
|\RNum{2}|
\leq C\int_{\mathbb{T}^2}|\nabla u| |\nabla^{2}\eta|^{4}\ud x\leq C|\nabla u|_{\infty} |\nabla^{2}\eta|_{4}^{4}.
\end{align*}
Besides, from the integration by parts and the fact that $u=K*\xi$ is divergence free, we get
\begin{align*}
\RNum{3}
=&-\sum_{i,j,k}\int_{\mathbb{T}^2}  \partial_{k}[u_{k} \partial_{ij} \eta|\nabla^{2} \eta|^{2}]\partial_{ij} \eta\ud x\\
=&-\sum_{i,j,k}\int_{\mathbb{T}^2}  \partial_{k}u_{k} |\partial_{ij} \eta|^{2}|\nabla^{2} \eta|^{2}+u_{k} \partial_{ijk} \eta|\nabla^{2} \eta|^{2}\partial_{ij}\eta +u_{k} |\partial_{ij} \eta|^{2} \partial_{k}|\nabla^{2} \eta|^{2}\ud x\\
=&-\sum_{i,j,k}\int_{\mathbb{T}^2}  u_{k} \partial_{ijk} \eta|\nabla^{2} \eta|^{2}\partial_{ij}\eta\ud x -\frac{1}{2}\sum_{k}\int_{\mathbb{T}^2} u_{k}  \partial_{k}|\nabla^{2} \eta|^{4}\ud x\\
=&-\sum_{i,j,k}\int_{\mathbb{T}^2}  u_{k} \partial_{ijk} \eta|\nabla^{2} \eta|^{2}\partial_{ij}\eta\ud x \\
=&-\RNum{3},
\end{align*}
and hence $\RNum{3}=0$. 
Collecting the above estimations  together, we have
\begin{align}\label{eta21}\notag
\frac{\ud}{\ud t}|\nabla^{2} \eta(t)|_{4}^{4}\le &
(8+C|\nabla u|_{\infty})|\nabla^{2}\eta|_{4}^{4} +C||\nabla W^{curl}||_{2,4}^{4}|\nabla^{2}u|_{\infty}^{4}+C
|\nabla^{2} u|^4_{\infty}|\nabla\eta|^4_{4}\\
\le &C(1+|\nabla u|_{\infty})|\nabla^{2}\eta|_{4}^{4} +C|\nabla^{2}u|_{\infty}^{4}
\end{align}
holds for some $C=C(\omega,\xi(0),T)$, where Assumption \ref{assum} was used in the second inequality.

Now we proceed to estimate $|\nabla^{2} u|_{\infty}$.
According to the embedding $H^{s}(\mathbb{T}^2)\hookrightarrow W^{2,\infty}(\mathbb{T}^2)$ for any $s>3$, it is sufficient to estimate $\|u\|_{s}$. Define $v=u-W$, and let $\alpha=(\alpha_1,\alpha_2)$ with $|\alpha|\le s$ being arbitrarily fixed. Then applying $\partial^{\alpha}$ to \eqref{Euler} gives
\begin{align}
\partial_{t} \partial^{\alpha}v+\partial^{\alpha}[(v \cdot \nabla) v]+\partial^{\alpha}[(W \cdot \nabla) v+(v \cdot \nabla) W]=-\partial^{\alpha}[\nabla \bm{\pi}+(W \cdot \nabla) W].  \label{dpde}
\end{align}
By taking the scalar product of (\ref{dpde}) against $\partial^{\alpha}v$ and $\nabla\cdot v=0$, we obtain
\begin{align*}
\frac{1}{2}\partial_{t} \left|\partial^{\alpha}v\right|_{2}^{2}=&- \left\langle \partial^{\alpha}[(W \cdot \nabla) W],\partial^{\alpha}v\right\rangle-\left\langle \partial^{\alpha}[(W \cdot \nabla) v],\partial^{\alpha}v\right\rangle\\
&-\left\langle \partial^{\alpha}[(v \cdot \nabla) W],\partial^{\alpha}v\right\rangle+\left\langle(v \cdot \nabla) \partial^{\alpha}v-\partial^{\alpha}[(v \cdot \nabla) v],\partial^{\alpha}v\right\rangle,
\end{align*}
The commutator estimate \eqref{commutator} implies that 
\begin{align}
\left|\partial^{\alpha}(v \cdot \nabla v)-v \cdot \nabla \partial^{\alpha} v\right|_{2} \leq C\|v\|_{s}|\nabla v|_{\infty} \label{keyest}
\end{align}
and
\begin{align*}
\left|\left\langle \partial^{\alpha}[(W \cdot \nabla) v],\partial^{\alpha}v\right\rangle\right|=&\left|\left\langle \partial^{\alpha}[(W \cdot \nabla) v]-(W \cdot \nabla) \partial^{\alpha}v,\partial^{\alpha}v\right\rangle\right| \notag\\
\leq& C\left( \|W\|_{s}|\nabla v|_{\infty}+\|W\|_{1,\infty}\| v\|_{s}\right)\left|\partial^{\alpha}v\right|_{2},
\end{align*}
because the divergence of $W$ vanishes.
In addition, the Young inequality yields
\begin{align*}
|\left\langle \partial^{\alpha}[(W \cdot \nabla) W],\partial^{\alpha}v\right\rangle|\leq \frac{1}{2} \left| \partial^{\alpha}v\right|_{2}^{2}+\frac{1}{2} \left| \partial^{\alpha}[(W \cdot \nabla) W]\right|_{2}^{2}, \end{align*}
and the Moser estimate  \eqref{Moser} gives
\begin{align}
|\left\langle \partial^{\alpha}[(v \cdot \nabla) W],\partial^{\alpha}v\right\rangle|\leq C\left(\|v\|_{s}^{2}|\nabla W|_{\infty}+|v|_{\infty}\|\nabla W\|_{s}\left| \partial^{\alpha}v\right|_{2}\right). \label{thirdterm}
\end{align}
Combining \eqref{keyest}-\eqref{thirdterm} and summing over all $|\alpha|\leq s$, it follows that
\begin{align*}
\frac{1}{2}\partial_{t} \left\|v\right\|_{s}^{2} \leq&\frac{1}{2} \left\| v\right\|_{s}^{2}+\frac{1}{2} \left\| (W \cdot \nabla) W\right\|_{s}^{2}+C\|v\|_{s}^{2}\left(\|W\|_{1,\infty}+|\nabla W|_{\infty}+|\nabla v|_{\infty}\right)\\
&+C\left( \|W\|_{s}|\nabla v|_{\infty}+|v|_{\infty}\|\nabla W\|_{s}\right)\|v\|_{s}\\
\leq&C\left[\left\| W\right\|_{s}^{2} (\left| \nabla W\right|_{\infty}^{2}+|\nabla v|_{\infty}^{2})+\left\|  W\right\|_{s+1}^{2}(\left| W\right|_{\infty}^{2} +|v|_{\infty}^{2})\right]\\
&+C^{\prime}\left(1+\|W\|_{1,\infty}+|\nabla W|_{\infty}+|\nabla v|_{\infty}\right)\|v\|_{s}^{2}\\
=:&F(t)+A(t)\|v\|_{s}^{2}.
\end{align*}
We make use of the Beale–Kato–Majda type inequality (see \cite[Formula (20)]{MR4107937})
$$
|\nabla u|_{\infty}\leq C|\xi|_{\infty}\left[1+\log \left(1+\frac{|\nabla \xi|_{4}}{|\xi|_{\infty}}\right)\right]
$$
to deduce that $|\nabla v|_{\infty}\leq |\nabla u|_{\infty}+|\nabla W|_{\infty}\le C, \ a.s.$
Taking $s=(h-1)\land (m+1)$, then Assumption \ref{assum} and Sobolev embedding theorem ensure that $A$ and $F$ belong to $ L^\infty([0,T]),\ a.s.$ The regularity of $\xi(0)$ indicates that $u(0) \in [H^{m+1}(\mathbb{T}^2)]^2\cap \mathcal{H}$. Hence,  the differential form of Gronwall lemma produces that for any  $t\in[0,T]$,
\begin{align*}
\|v(t)\|_{s}^{2}\leq \exp \left(\int_{0}^{t} A\left(s\right) \ud s\right) \|u(0)\|_{s}^{2}+\int_{0}^{t} \exp \left(\int_{s}^{t} A\left(u\right)\ud u\right) F(s) \ud s\le C,\ a.s.
\end{align*}
 Recalling $u=v+W$, thus we have shown that 
$$\sup_{t\in[0,T]}|\nabla^{2} u(t)|_{\infty}\le C\sup_{t\in[0,T]}\|u(t)\|_{s}\le C,\ a.s.$$ 
Substituting the above estimations into  \eqref{eta21}, we get by using Gronwall lemma 
$$
\sup _{0\leq t \leq T}|\nabla^{2} \eta(t)|_{4} \leq C\left(T,\xi(0),\omega \right).
$$
Finally, the regularity of the $W$ and the fact $\xi=\eta+W^{c u r l}$ conclude the proof.
\end{proof}

With Proposition \ref{W24} in hand, it is easy to derive the time Hölder continuity of the exact solution in $H^1(\mathbb{T}^2)$.
\begin{cor}\label{Hold}
Under the same conditions of Proposition \ref{W24}, the exact solution $\xi(t)$ of \eqref{vorticity} in $H^1(\mathbb{T}^2)$ is $\theta$-H\"older continuous in time with $\theta\in(0,\frac{1}{2})$, i.e., for any $0\le s<t\le T$,
$$\|\xi(t)-\xi(s)\|_1\le C(\xi(0),T,\theta,\omega)|t-s|^{\theta}.$$

\end{cor}
\begin{proof}
The exact solution of \eqref{vorticity} satisfies
\begin{align}
\xi(t)-\xi(s)=-\int_{s}^{t}K*\xi(r)\cdot \nabla \xi(r)\mathrm{d}r+\int_{s}^{t}\mathrm{d}W^{curl}(r), 
\label{inteq}
\end{align}
for any $t>s\ge0$.
It follows from Assumption \ref{assum} that almost surely sample paths of $W^{curl}$ in $H^1(\mathbb{T}^2)$ are H\"older continuous with exponent $\theta\in (0,\frac{1}{2})$ on $[0, T]$, i.e.,
\begin{align}\label{e2}
\left\|\int_{s}^{t}\ud W^{curl}(r)\right\|_{1}=\left\|W^{curl}(t)-W^{curl}(s)\right\|_{1}\leq C(\omega,\theta)|t-s|^{\theta}. 
\end{align}
As a result of H\"older's inequality, \eqref{K4}, and Proposition \ref{W24}, it holds that
\begin{align*}
\left|\int_{s}^{t}K*\xi(r)\cdot \nabla\xi(r)\ud r\right|_{2}
\leq&\sup_{r\in[0,T]}\left|K*\xi(r)\right|_{4}\sup_{r\in[0,T]}\left\| \xi(r)\right\|_{1,4}(t-s)\notag\\
\leq&C\sup_{r\in[0,T]}\left|\xi(r)\right|_{2}\sup_{r\in[0,T]}\left\| \xi(r)\right\|_{1,4} (t-s)\notag\\
\leq&C(\xi(0),T,\omega)(t-s). 
\end{align*}
Similarly, by chain rule, Lemma \ref{Kbound}, \eqref{K4}, and Proposition \ref{W24},
\begin{align}\label{H1-hold}\notag
\left|\nabla\int_{s}^{t}K*\xi(r)\cdot \nabla\xi(r)\ud r\right|_{2}
\leq&\int_{s}^{t}(\left|\nabla [K*\xi(r)]\right|_{4}\left| \nabla\xi(r)\right|_{4}+\left|K*\xi(r)\right|_{4}\left|\nabla^{2}\xi(r)\right|_{4})\ud r\notag\\\notag
\leq&\int_{s}^{t}(\left|\xi(r)\right|_{4}\left| \nabla\xi(r)\right|_{4}+\left|\xi(r)\right|_{2}\left|\nabla^{2}\xi(r)\right|_{4})\ud r\notag\\
\leq&C(T,\xi(0),\omega)(t-s). 
\end{align}
Combining the above estimates and \eqref{inteq}, the desired result follows.
\end{proof}

For $i=0,\ldots,n$, let $e(t_{i}):=\xi_{i}-\xi(t_{i}).$ Theorem \ref{convergenceorder} states that the pathwise convergence order of $|e(t_i)|_2$ is almost $\frac{1}{2}$, whose proof relies on Proposition \ref{W24}, Corollary \ref{Hold}, and the orthogonality property $\langle K*\xi\cdot\nabla e(t_{i}),e(t_{i})\rangle=0$ of the nonlinear term.

\begin{tho} \label{convergenceorder}
Let Assumption \ref{assum} hold. If 
$\xi(0)\in H^{m}(\mathbb{T}^2)\cap\mathcal E$ almost surely  for some $m\ge\frac{5}{2}$ and $\xi(0)\in L^{p}(\Omega,L^{p}(\mathbb{T}^2))$ for all large enough $p\in\mathbb{N}$, then for sufficiently small $\tau>0$, 
$$
\sup_{1\leq i\leq n}\left|\xi_{i}-\xi(t_{i})\right|_{2} \leq C(\xi(0),T,\omega,\theta)\tau^{\theta},\  a.s.,
$$
where $\theta \in (0,\frac{1}{2})$.
\end{tho}

\begin{proof}   
By comparing \eqref{inteq} and \eqref{numericalsol}, we infer that the error  $e(t_{i+1})$ satisfies 
\begin{align}\label{IJ}\notag
e(t_{i+1})
&=\xi_{i}-\tau (K*\xi_{i})\cdot \nabla \Phi_{\tau}(\xi_{i})-\xi(t_{i})+\int_{t_{i}}^{t_{i+1}}K*\xi(s)\cdot \nabla \xi(s)\mathrm{d}s\\\notag
&=e(t_{i})+\int_{t_{i}}^{t_{i+1}}K*\xi(s)\cdot \nabla \xi(s)-(K*\xi_{i})\cdot \nabla \Phi_{\tau}(\xi_{i})\mathrm{d}s\\
&=:e(t_{i})+I^i+J^i,
\end{align}
for $i=0,\ldots,n-1$, where
\begin{align*}
I^i&=\int_{0}^{\tau}K*(\xi(t_{i}+r)-\xi_{i})\cdot \nabla\xi(t_{i}+r)\ud r,\\
J^i&=\int_{0}^{\tau}(K*\xi_{i})\cdot \nabla[\xi(t_{i}+r)-\Phi_{\tau}(\xi_{i})]\ud r.
\end{align*}
Propositions \ref{lm3.2} and \ref{W24} ensure that $e(t_{i+1})\in L^2(\mathbb{T}^2)$, a.s., which allows us to apply the scalar product against $e(t_{i+1})$ to get
\begin{align}
|e(t_{i+1})|_{2}^{2}=&\left \langle e(t_{i})+I^i+J^{i},e(t_{i+1}) \right \rangle\notag\\
\leq&\frac{1}{2}|e(t_{i+1})|_{2}^{2}+\frac{1}{2}|e(t_{i})|_{2}^{2}+\left \langle I^i+J^i,e(t_{i+1}) \right \rangle.\label{ierror}
\end{align}
To estimate $\left \langle I^i,e(t_{i+1}) \right \rangle$, we divide it into two parts: 
\begin{align}
\left \langle I^i,e(t_{i+1}) \right \rangle
=&\int_{0}^{\tau}\left \langle K*(\xi(t_{i}+r)-\xi(t_{i}))\cdot \nabla\xi(t_{i}+r),e(t_{i+1}) \right \rangle\ud r\notag\\
&-\int_{0}^{\tau}\left \langle (K*e(t_{i}))\cdot \nabla\xi(t_{i}+r),e(t_{i+1}) \right \rangle\ud r\notag\\
=&:I^i_{1}+I^i_{2}.
\label{error1}
\end{align}
For the term $I_{1}^i$, we use inequality \eqref{K4}, Proposition \ref{W24}, and Corollary \ref{Hold} to derive that
\begin{align*}
|I_{1}^i|
\leq&\frac{  \tau}{2}|e(t_{i+1})|_{2}^{2}+\frac{1}{2 }\int_{0}^{\tau}|K*(\xi(t_{i}+r)-\xi(t_{i}))|_{4}^{2} |\nabla\xi(t_{i}+r)|_{4}^{2}\ud r\\
\leq&\frac{  \tau}{2}|e(t_{i+1})|_{2}^{2}+C\int_{0}^{\tau}|\xi(t_{i}+r)-\xi(t_{i})|_{2}^{2}\ud r\\
\leq&\frac{  \tau}{2}|e(t_{i+1})|_{2}^{2}+C \tau^{2\theta+1},
\end{align*}
for some $C=C(\xi(0),T,\omega,\theta)$.
Similarly, we also have
\begin{align*}
|I_{2}^i|\leq&\frac{  \tau}{2}|e(t_{i+1})|_{2}^{2}+\frac{1}{2 }\int_{0}^{\tau}\left|K*e(t_{i})\right|^2_{4}\left| \nabla\xi(t_{i}+r)\right|^2_{4}\ud r\\
\leq&\frac{  \tau}{2}|e(t_{i+1})|_{2}^{2}+C\int_{0}^{\tau}\left|e(t_{i})\right|^2_{2} \left\|\xi(t_{i}+r)\right\|_{1,4}^2\ud r\\
\leq&C\tau\left|e(t_{i})\right|_{2}^{2}+\frac{ \tau}{2}\left|e(t_{i+1}) \right|_{2}^{2}.
\end{align*}
The above estimates show
\begin{align}
|\left \langle I^i,e(t_{i+1}) \right \rangle|\leq\tau\left|e(t_{i+1}) \right|_{2}^{2}+C\tau\left|e(t_{i})\right|_{2}^{2}+C\tau^{2\theta+1}. \label{2}
\end{align}

It remains to estimate $\left \langle J^i,e(t_{i+1}) \right \rangle$. Observe that by \eqref{numericalsol} and \eqref{inteq} 
\begin{align}
\left \langle J^i,e(t_{i+1}) \right \rangle=&\int_{0}^{\tau}\left \langle (K*\xi_{i})\cdot \nabla[\xi(t_{i}+r)-\Phi_{\tau}(\xi_{i})],e(t_{i+1}) \right \rangle\ud r\notag\\
=&\int_{0}^{\tau}\left \langle (K*\xi_{i})\cdot \nabla[\xi(t_{i}+r)-\Phi_{\tau}(\xi_{i})+e(t_{i+1})],e(t_{i+1}) \right \rangle\ud r\notag\\
=&\int_{0}^{\tau}\left \langle (K*\xi_{i})\cdot \nabla\int_{t_{i}+r}^{t_{i+1}}K*\xi(s)\cdot \nabla\xi(s)\ud s,e(t_{i+1}) \right \rangle\ud r\notag\\
&+\int_{0}^{\tau}\left \langle (K*\xi_{i})\cdot \nabla\int_{t_{i}}^{t_{i}+r}\ud W^{curl}(s),e(t_{i+1}) \right \rangle\ud r\notag\\
=&:J_{1}^i+J_{2}^i.\label{error2}
\end{align}
Recall that $K\in L^\frac{4}{3}(\mathbb T^2)$. By Proposition \ref{lm3.2} with $\varepsilon=\frac{1}{2}-\theta$ and estimate \eqref{H1-hold},
\begin{align*}\notag
|J_{1}^i|
&\leq\frac{  \tau}{2}|e(t_{i+1})|_{2}^{2}+\frac{1}{2 }\int_{0}^{\tau}|K*\xi_{i}|_{\infty}^{2} \Big|\nabla\int_{t_{i}+r}^{t_{i+1}}K*\xi(s)\cdot \nabla\xi(s)\ud s\Big|_{2}^{2}\ud r\\\notag
&\leq\frac{  \tau}{2}|e(t_{i+1})|_{2}^{2}+C(T,\xi(0),\omega)\tau^{3}|\xi_{i}|_{4}^{2}\\
&\leq\frac{  \tau}{2}|e(t_{i+1})|_{2}^{2}+C(T,\xi(0),\omega,\theta)\tau^{2+2\theta}. 
\end{align*}
Similarly, the term $J_{2}^i$ is bounded as
\begin{align*}F
|J_{2}^i|
\leq \frac{ \tau}{2}|e(t_{i+1})|_{2}^{2}+C(T,\xi(0),\omega,\theta)\tau^{4\theta},
\end{align*}
thanks to \eqref{e2}.
Therefore, we conclude from the above estimates on $J^i_{1}$ and $J^i_{2}$ that
\begin{align*}
|\left \langle J^i,e(t_{i+1}) \right \rangle|\leq\tau|e(t_{i+1})|_{2}^{2}+C(T,\xi(0),\omega,\theta)\tau^{4\theta},
\end{align*}
which together with \eqref{ierror}, \eqref{2} gives that for any $i=0,\ldots,n-1$,
\begin{align*}
(1-4\tau)|e(t_{i+1})|_{2}^{2}\le(1+C\tau)|e(t_{i})|_{2}^{2}+C\tau^{4\theta}.
\end{align*}
Choosing sufficiently small $\tau>0$, the desired result follows from the discrete Gronwall lemma.
\end{proof}
\subsection{Convergence analysis in probabilty}
Utilizing Theorem \ref{convergenceorder} above, we conclude the following error estimation in probability sense, which shows that the order of convergence in probability of the splitting SIE method is nearly 1.
\begin{tho} \label{order1inprothm}
With the conditions of Theorem \ref{convergenceorder}, for any $\beta\in (0,1)$, we have
\begin{align}
\mathbb{P}\Big(\sup_{1\leq i \leq n}|e(t_{i})|_{2} \geq \tau^{\beta}\Big)\to 0,\ as\   \tau \to 0.  \label{order1inpro}
\end{align}
\end{tho}
\begin{proof}
In the following we assume that $\tau>0$ and $\nu>0$ are small enough.

\emph{Step 1:} We prove \eqref{order1inpro} for $\beta\in (0,\frac{3}{4})$.

Rearranging \eqref{ierror}, \eqref{error1}, and \eqref{error2}, taking $\theta =\frac{1-\nu}{2}$, we get for  $i=0,\ldots,n-1$,
\begin{align*}
(1-4\tau)|e(t_{i+1})|_{2}^{2}\le &(1+C\tau)|e(t_{i})|_{2}^{2}+2\int_{0}^{\tau}\left \langle (K*\xi_{i})\cdot \nabla\int_{t_{i}}^{t_{i}+r}\ud W^{curl}(s),e(t_{i+1}) \right \rangle\\
&+\left \langle K*\int_{t_{i}}^{t_{i}+r}\ud W^{curl}(s)\cdot \nabla\xi(t_{i}+r),e(t_{i+1}) \right \rangle\ud r+C_{\nu}\tau^{3-\nu}. 
\end{align*}
Let $l\in\{1,\ldots,n\}$ be arbitrarily fixed. Using the notation $M:=(4+C)/(1-4\tau)$,  we get
\begin{align*}
|e(t_{l})|_{2}^{2}\le &\frac{2}{1-4\tau}\sum_{i=0}^{l-1}\left(1+M\tau\right)^{l-i-1}\left[\int_{0}^{\tau}\left \langle (K*\xi_{i})\cdot \nabla\int_{t_{i}}^{t_{i}+r}\ud W^{curl}(s),e(t_{i+1}) \right \rangle\ud r\right.\\
&\left.+\int_{0}^{\tau}\left \langle K*\int_{t_{i}}^{t_{i}+r}\ud W^{curl}(s)\cdot \nabla\xi(t_{i}+r),e(t_{i+1}) \right \rangle\ud r\right]+C(\nu)\tau^{2-\nu}\\
=&:\frac{2}{1-4\tau}(S^l+R^l)+C(\nu)\tau^{2-\nu}.
\end{align*}
By noticing $W=K*W^{curl}$, we split $R^l$ into three terms:
\begin{align*}
R^l
=&\sum_{i=0}^{l-1}\left(1+M\tau\right)^{l-i-1}\int_{t_{i}}^{t_{i+1}}\left \langle \int_{t_{i}}^{r}\ud W(s)\cdot \nabla\xi(r),e(t_{i+1}) \right \rangle\ud r\\
=&\sum_{i=0}^{l-1}\left(1+M\tau\right)^{l-i-1}\int_{t_{i}}^{t_{i+1}}\left \langle \int_{t_{i}}^{r}\ud W(s)\cdot \nabla\xi(t_{i}),e(t_{i}) \right \rangle\ud r\\
&+\sum_{i=0}^{l-1}\left(1+M\tau\right)^{l-i-1}\int_{t_{i}}^{t_{i+1}}\left \langle \int_{t_{i}}^{r}\ud W(s)\cdot \nabla[\xi(r)-\xi(t_{i})],e(t_{i}) \right \rangle\ud r\\
&+\sum_{i=0}^{l-1}\left(1+M\tau\right)^{l-i-1}\int_{t_{i}}^{t_{i+1}}\left \langle \int_{t_{i}}^{r}\ud W(s)\cdot \nabla\xi(r),e(t_{i+1})-e(t_{i}) \right \rangle\ud r\\
=&:R^l_{1}+R^l_{2}+R^l_{3},
\end{align*} 
and estimate each term separately.
 
Using the stochastic Fubini theorem, we rewrite $R^l_{1}$ as a stochastic integral
\begin{align*}
R^l_{1}=&\sum_{i=0}^{l-1}\left(1+M\tau\right)^{l-i-1}\int_{t_{i}}^{t_{i+1}}\left \langle (t_{i+1}-s)e(t_{i}) , \nabla\xi(t_{i})\cdot\ud W(s)\right \rangle\\
=&\int_{0}^{t_{l}}\left \langle \left(1+M\tau\right)^{l-\frac{\eta_n(s) }{\tau}}(\eta_n(s) -s)e(\kappa_n(s) ) , \nabla\xi(\kappa_n(s) )\cdot\ud W(s)\right \rangle\\
=&:\int_{0}^{t_{l}}\phi_l(s)\ud W(s),
\end{align*}
where the process $\phi_l=\{\phi_l(s)\}_{s\in[0,T]}$  given by
\begin{align*}
\phi_l(s)h=\left \langle \left(1+M\tau\right)^{l-\frac{\eta_n(s) }{\tau}}(\eta_n(s) -s)e(\kappa_n(s) ) , \nabla\xi(\kappa_n(s) )\cdot h\right \rangle\mathbf{1}_{\{s\le t_l\}} ,\ \forall\ h\in \mathcal{H}
\end{align*}
is adapted to $\{\mathcal{F}_s\}_{s\in [0,T]}$. We claim that $\phi_l\in L^{\infty}(0,T;L_{2}(\mathcal{H}_{0},\mathbb{R})),$ a.s., where $\mathcal{H}_{0}=Q^{\frac{1}{2}}(\mathcal{H})$ is the Cameron--Martin space of $W$.  Indeed, for any $s\in[0,T]$,
\begin{align*}
&||\phi_l(s)||_{ L_{2}(\mathcal{H}_{0},\mathbb{R})}^{2}\\=&\sum_{k\in \mathbb{Z}_{0}^{2}}\left \langle \left(1+M\tau\right)^{l-\frac{\eta_n(s) }{\tau}}(\eta_n(s) -s)e(\kappa_n(s) ) , \nabla\xi(\kappa_n(s) )\cdot Q^{\frac{1}{2}}g_{k}\right \rangle^2\mathbf{1}_{\{s\le t_l\}}\\
\le &\sum_{k\in \mathbb{Z}_{0}^{2}}\left(1+M\tau\right)^{2\left(l-\frac{\eta_n(s) }{\tau}\right)}(\eta_n(s) -s)^{2}|e(\kappa_n(s) )|_{2}^{2}c_{k}^{2}|\nabla\xi(\kappa_n(s) )\cdot g_{k}|_{2}^{2}\\
\le &C(\omega,\xi(0),T,\nu)\tau^{3-\nu}\left(\sup_{s\in[0,T]}\|\xi(s)\|_{1,4}^{2}\right)\sum_{k\in \mathbb{Z}_{0}^{2}}c_{k}^{2} |g_{k}|_{4}^{2}\\
\le &C(\omega,\xi(0),T,\nu)\tau^{3-\nu}<\infty,\ a.s.,
\end{align*}
thanks to Assumption \ref{assum}, Proposition \ref{W24}, and Theorem \ref{convergenceorder}.
Thus, 
\begin{align*}
\mathbb{P}\left(\int_{0}^{T}||\phi_l(s)||_{ L_{2}(\mathcal{H}_{0},\mathbb{R})}^{2}\ud s<\infty\right)=1,
\end{align*}
from which we deduce that $\{\int_{0}^{t}\phi_l(s)\ud W(s)\}_{t\in[0,T]}$ is a continuous real-valued local martingale. This fact allows us to apply Lemma \ref{LRineq} to obtain
\begin{align}\notag
\mathbb{P}\left(\sup_{1\le l\le n}|R^l_{1}|\ge \tau^{\frac{3}{2}-2\nu}\right)
\le &\mathbb{P}\left(\sup_{t\in[0,T]}\left|\int_{0}^{t}\phi_l(s)\ud W(s)\right|\ge \tau^{\frac{3}{2}-2\nu}\right)\\
\le &3\tau^{\nu}+ \mathbb{P}\left(\left[\int_{0}^{T}||\phi_l(s)||_{ L_{2}(\mathcal{H}_{0},\mathbb{R})}^{2}\ud s\right]^{\frac{1}{2}}\ge \tau^{\frac{3}{2}-\nu}\right)\notag\\
\le &3\tau^{\nu}+ \mathbb{P}\left( C(\omega,T,\xi_{0})\tau^{\frac{\nu}{2}}\ge 1\right) \to 0, \ as \ \tau \to 0. \label{coninp}
\end{align}

To estimate the term $R^l_{2}$, we notice that, analogue to \eqref{e2}, Assumption \ref{assum} and the embedding relation $H^{2}(\mathbb{T}^2)\hookrightarrow L^{\infty}(\mathbb{T}^2)$ imply
\begin{equation}\label{hol-noi}
\left| \int_{t_{i}}^{r}\ud W(s)\right|_{\infty}\le C\left\| \int_{t_{i}}^{r}\ud W(s)\right\|_{2}\le C(\nu,\omega)(r-t_i)^{\frac{1-\nu}{2}}.
\end{equation} 
Therefore, according to \eqref{hol-noi},  Corollary \ref{Hold}, and Theorem \ref{convergenceorder}, 
\begin{align*}
|R^l_{2}|
\le&C\sum_{i=0}^{l-1}\int_{t_{i}}^{t_{i+1}}\left| \int_{t_{i}}^{r}\ud W(s)\right|_{\infty}\left|\nabla[\xi(r)-\xi(t_{i})]\right|_{2}\left|e(t_{i})\right|_{2}\ud r
\le C(\omega,T,\xi(0),\nu)\tau^{\frac{3-3\nu}{2}},
\end{align*}
for any $l=1,\ldots,n$.
Recalling \eqref{IJ}, we obtain by the estimates on $I^i$ and $J^i$ in the proof of Theorem \ref{convergenceorder} that 
\begin{align}\label{I+J}
|I^i|_2+|J^i|_2\le C(\nu,\omega,\xi(0),T)\tau^{\frac{3-\nu}{2}},
\end{align} 
from which we derive that for any $l=1,\ldots,n$,
\begin{align*}
|R^l_{3}|\le& C\sum_{i=0}^{l-1}\int_{t_{i}}^{t_{i+1}}\left|\left \langle \int_{t_{i}}^{r}\ud W(s)\cdot \nabla\xi(r),I^i+J^i\right \rangle\right|\ud r
\le C(\omega,T,\xi(0),\nu)\tau^{2-\nu}.
\end{align*}

For the term $S^l$, we decompose it  into two parts $S^l=S^l_{1}+S^l_{2}$, where
\begin{align*}
S^l_{1}=&\frac{1}{1-C\tau}\sum_{i=1}^{l-1}\left(1+M\tau\right)^{l-i-1}\int_{t_{i}}^{t_{i+1}}\Big \langle (K*\xi_{i})\cdot \nabla\int_{t_{i}}^{r}\ud W^{curl}(s),e(t_{i}) \Big \rangle\ud r,\\
S^l_{2}=&\frac{1}{1-C\tau}\sum_{i=1}^{l-1}\left(1+M\tau\right)^{l-i-1}\int_{t_{i}}^{t_{i+1}}\Big \langle (K*\xi_{i})\cdot \nabla\int_{t_{i}}^{r}\ud W^{curl}(s),e(t_{i+1})-e(t_{i}) \Big \rangle\ud r.
\end{align*}
And we proceed to utilize the similar arguments as in the estimations of $R^l_1$ and $R^l_3$ to deal with $S^l_{1}$ and $S^l_{2}$, respectively. More precisely, for any $f\in H^{1}(\mathbb{T}^2)\cap\mathcal{E}$, let us define the $\{\mathcal{F}_s\}_{\,s\in [0,T]}$-adapted process $\psi_l=\{\psi_l(s)\}_{s\in[0,T]}$ by 
\begin{align*}
\psi_l(s)f=\left \langle \left(1+M\tau\right)^{l-\frac{\eta_n(s) }{\tau}}(\eta_n(s) -s)e(\kappa_n(s) ) , (K*\xi_{\frac{\kappa_n(s) }{\tau}})\cdot \nabla f\right \rangle\mathbf{1}_{\{s\le t_l\}}.
\end{align*}
Then using stochastic Fubini theorem again, we reformulate $S^l_{1}$ as follows:
\begin{align*}
S^l_{1}
=&\frac{1}{1-C\tau}\sum_{i=1}^{l-1}\left(1+M\tau\right)^{l-i-1}\left \langle \int_{t_{i}}^{t_{i+1}}(t_{i+1}-s) e(t_{i}),(K*\xi_{i})\cdot \nabla\ud W^{curl}(s) \right \rangle\\
=&\frac{1}{1-C\tau}\left \langle \int_{0}^{t_{l}}\left(1+M\tau\right)^{l-\frac{\eta_n(s) }{\tau}}(\eta_n(s) -s) e(\kappa_n(s) ),(K*\xi_{\frac{\kappa_n(s) }{\tau}})\cdot \nabla\ud W^{curl}(s) \right \rangle\\
=&\frac{1}{1-C\tau}\int_{0}^{t_{l}}\psi_l(s)\ud W^{curl}(s).
\end{align*}
Recall that $W^{curl}$ is an $H^{1}(\mathbb{T}^2)\cap \mathcal{E}$-valued $Q_{1}$-Wiener process  whose Cameron--Martin space is denoted by $\mathcal{E}_{0}=Q_{1}^{\frac{1}{2}}(H^{1}(\mathbb{T}^2)\cap\mathcal{E})$.
Assumption \ref{assum} and Theorem \ref{convergenceorder} yield
\begin{align*}
&||\psi_l(s)||_{ L_{2}(\mathcal{E}_{0},\mathbb{R})}^{2}\\=&\sum_{k\in \mathbb{Z}_{0}^{2}}\left \langle \left(1+M\tau\right)^{l-\frac{\eta_n(s) }{\tau}}(\eta_n(s) -s)e(\kappa_n(s) ) , (K*\xi_{\frac{\kappa_n(s) }{\tau}})\cdot \nabla Q_{1}^{\frac{1}{2}}e_{k}\right \rangle^{2}\mathbf{1}_{\{s\le t_l\}}\\
\le &e^{MT}\tau^2\sum_{k\in \mathbb{Z}_{0}^{2}}|e(\kappa_n(s) )|_{2}^{2}4\pi^{2}|k|^{2}c_{-k}^{2}|(K*\xi_{\frac{\kappa_n(s) }{\tau}})\cdot \nabla e_{k}|_{2}^{2}\\
\le &C(\omega,\xi(0),T)\tau^{3-\nu}|K*\xi_{\frac{\kappa_n(s) }{\tau}}|_{\infty}^{2},\ a.s., \ \forall\, s\in[0,T].
\end{align*}
Notice that Young inequality for convolution together with Proposition \ref{lm3.2} indicates 
\begin{align}\label{convolu}
|K*\xi_{\frac{\kappa_n(s) }{\tau}}|_{\infty}\le|K|_{\frac{4}{3}}|\xi_{\frac{\kappa_n(s) }{\tau}}|_4\le C(\omega,\xi(0),T,\nu)\tau^{-\nu}.
\end{align}
Thus, it holds that
\begin{align*}
\mathbb{P}\left(\int_{0}^{T}||\psi_l(s)||_{ L_{2}(\mathcal{E}_{0},\mathbb{R})}^{2}\ud s<\infty\right)=1,
\end{align*}
and $\{\int_{0}^{t}\psi_l(s)\ud W^{curl}(s)\}_{t\in[0,T]}$ is a continuous real-valued local martingale. Similar to \eqref{coninp}, we arrive at
\begin{align*}
\mathbb{P}\left(\sup_{1\le l\le n}|S^l_{1}|\ge \tau^{\frac{3}{2}-3\nu}\right)\le&\mathbb{P}\left(\sup_{t\in[0,T]}\left|\int_{0}^{t}\psi_l(s)\ud W(s)\right|\ge \tau^{\frac{3}{2}-3\nu}\right)\\\le &3\tau^{\nu}+ \mathbb{P}\left(\left[\int_{0}^{T}||\psi_l(s)||_{ L_{2}(\mathcal{E}_{0},\mathbb{R})}^{2}\ud s\right]^{\frac{1}{2}}\ge \tau^{\frac{3}{2}-2\nu}\right)\\
\le &3\tau^{\nu}+ \mathbb{P}\left( C(\omega,T,\xi(0))\tau^{\frac{\nu}{2}}\ge 1\right) \to 0, \ as \ \tau \to 0.
\end{align*}
Similar to $R^l_{3}$, one can also get from  \eqref{I+J} and \eqref{convolu} that for $l=1,\ldots,n$,
\begin{align*}
|S^l_{2}|\le C\sum_{i=0}^{l-1}\int_{t_{i}}^{t_{i+1}}\left|\left \langle (K*\xi_{i})\cdot \nabla\int_{t_{i}}^{r}\ud W^{curl}(s),I^i+J^i \right \rangle\right|\ud r\le C(\omega,T,\xi(0),\nu)\tau^{2-2\nu}.
\end{align*}
Combining the above estimates together, we conclude that
\begin{align}
\mathbb{P}\Big(\sup_{1\leq l \leq n}|e(t_{l})|_{2} \geq \tau^{\frac{3-6\nu}{4}}\Big)\to 0,\ as\   \tau \to 0, \label{convergenceinpro2}
\end{align}
which completes the proof of \emph{Step 1}.

\emph{Step 2:} We prove \eqref{order1inpro} for $\beta\in [\frac{3}{4},1)$.

First,  the convergence order in probability can be improved to be close to $\frac{7}{8}$ by utilizing \eqref{convergenceinpro2} to reestimate the terms $R^l_1$, $R^l_2$, and $S^l_1$.

For $R^l_1$, similar to \eqref{coninp}, except using \eqref{convergenceinpro2} instead of Theorem \ref{convergenceorder}, we have 
\begin{align*}
\mathbb{P}\left(\sup_{1\le l\le n}|R^l_{1}|\ge \tau^{\frac{7}{4}-2\nu}\right)\le
&3\tau^{\nu}+ \mathbb{P}\left(\left[\int_{0}^{T}||\phi_l(s)||_{ L_{2}(\mathcal{H}_{0},\mathbb{R})}^{2}\ud s\right]^{\frac{1}{2}}\ge\tau^{\frac{7}{4}-\nu}\right)\\
\le &3\tau^{\nu}+ \mathbb{P}\left(C(\omega,T,\xi(0),\nu)\tau\sup_{1\le i\le n}|e(t_{i})|_{2}\ge \tau^{\frac{7}{4}-\nu}\right)\\
\to& 0, \ as \ \tau \to 0.
\end{align*}
Similarly, it also holds that
\begin{align*}
\mathbb{P}\left(\sup_{1\le l\le n}|S^l_{1}|
\ge \tau^{\frac{7}{4}-2\nu}\right)\to 0, \ as \ \tau \to 0.
\end{align*}
Observe that
\begin{align*}
\mathbb{P}\left(|R^l_{2}|\ge \tau^{\frac{7}{4}-2\nu}\right)\le \mathbb{P}\left(C(\omega,T,\xi(0),\nu)\tau^{1-\nu}\sup_{1\le i\le n}|e(t_{i})|_{2}\ge \tau^{\frac{7}{4}-2\nu}\right)\to 0, \ as \ \tau \to 0.
\end{align*}
Combining the above estimates, we conclude that
\begin{align*}
\mathbb{P}\Big(\sup_{1\leq i \leq n}|e(t_{i})|_{2} \geq \tau^{\frac{7}{8}-\nu}\Big)\to 0,\ as\   \tau \to 0. 
\end{align*}
Finally, we finish the proof by repeating the above arguments.
\end{proof}

\begin{cor}
With the conditions of Theorem \ref{convergenceorder}, for any $\beta\in (0,1)$, we have
\begin{align*}
\mathbb{P}\Big(\sup_{1\leq i \leq n}[||u(t_{i})-u_{i}||_{1}+||\bm\pi(t_{i})-\bm\pi_{i}||_{1}]\geq \tau^{\beta}\Big)\to 0,\ as\   \tau \to 0.  
\end{align*}
\end{cor}
\begin{proof}
It follows immediately from Lemma \ref{Kbound} that
\begin{align*}
||u(t_{i})-u_{i}||_{1}=||K*e(t_{i})||_{1}\le C|e(t_{i})|_{2}.
\end{align*}
Poincáre inequality, Lemma \ref{Kbound}, and the fact $K\in L^{\frac{4}{3}}(\mathbb{T}^2)$ yield
\begin{align*}
||\bm\pi(t_{i})-\bm\pi_{i}||_{1}\le&C|\nabla\bm\pi(t_{i})-\nabla\bm\pi_{i}|_{2}\\
=&C|(\mathcal{P}-{\rm Id})[(u(t_{i}) \cdot \nabla) u(t_{i})-(u_i\cdot \nabla) u_{i}]|_2\\
\le&C|((u(t_{i})-u_i) \cdot \nabla) u(t_{i})+(u_i\cdot \nabla) (u(t_{i})-u_{i})|_2\\
\le&C[|u(t_{i})-u_i|_2 \|u(t_{i})\|_{1,\infty}+|u_i|_{\infty} |\nabla(u(t_{i})-u_{i})|_2]\\
\le&C(\omega,T,\xi_{0})|e(t_{i})|_2+C|K|_{\frac{4}{3}}|\xi_i|_{4} |e(t_{i})|_2\\
\le&C(\omega,T,\xi_{0},\varepsilon)(1+\tau^{-\varepsilon})|e(t_{i})|_2,
\end{align*}
for sufficiently small $\varepsilon>0$. Finally, Theorem \ref{order1inprothm} concludes the proof.
\end{proof}
\section{Future work}
The strong convergence rate of a numerical approximation for SEEs is a fundamental index to characterize the efficiency and accuracy of a numerical method, which has been widely studied in the case of stochastic semilinear evolution equations.  Due to the presence of the quadratic nonlinear term, if one aims to obtain the strong convergence rate of a certain numerical method for stochastic Euler equations,
the exponential moment estimates of both the  exact solution and the  numerical solutions may be required.
However, we are not aware of any existing results on exponential moment estimates, or even algebraic moment estimates of the exact solution. We leave the construction of a strong convergence numerical method for Eq. \eqref{Euler} for a future work, where the tamed technique, truncated technique or splitting technique may be employed.

\bibliographystyle{plain}
\bibliography{Euler-reference}

\end{document}